\documentclass[10pt]{smfart}

\usepackage[T1]{fontenc}
\usepackage[utf8]{inputenc}
\usepackage{soul}
\usepackage{amsfonts,amsmath,amssymb,amsthm}

\usepackage{enumitem}
\usepackage[french]{babel}
\usepackage[colorlinks=true]{hyperref}

\theoremstyle{plain}
\newtheorem{theoreme}{Th\'eor\`eme}[section]

\newtheorem{corollaire}[theoreme]{Corollaire}
\newtheorem{proposition}[theoreme]{Proposition}
\newtheorem{lemme}[theoreme]{Lemme}

\theoremstyle{definition}
\newtheorem{definition}[theoreme]{D\'efinition}

\theoremstyle{remark}
\newtheorem{remarque}[theoreme]{Remarque}
\newtheorem{remarques}[theoreme]{Remarques}

\newcommand{\AK}{\mathbf{A}_K^\times}
\newcommand{\Aconc}[1]{\mathbf{A}^{\times}_{K,#1}}
\newcommand{\Afin}{\mathbf{A}^{\times}_{K,\mathrm{f}}}
\newcommand{\Ainf}{\mathbf{A}^{\times}_{K,\infty}}
\newcommand{\blabla}{\varphi}
\newcommand{\blaloc}{\phi}
\newcommand{\C}{\mathbf{C}}

\newcommand{\cond}[1][f]{\mathfrak{#1}}

\newcommand{\End}{\mathrm{End}}
\newcommand{\entiers}[1][K]{\mathcal{O}_{#1}}
\newcommand{\F}{\mathbf{F}}
\newcommand*{\overbar}[1]{\mkern 3mu\overline{\mkern-3mu#1\mkern-3mu}\mkern 3mu} 
\newcommand{\FFbar}{\overbar{\mathbf{F_\ell}}} 
\newcommand{\FFstar}{{\FFbar}^{\mkern-4mu\times}} 
\newcommand{\Qellbar}{\overbar{\Q_\ell}}
\newcommand{\Zellbar}{\overbar{\Z_\ell}}
\newcommand{\Zellbarstar}{{\Zellbar}^{\mkern-4mu\times}}
\newcommand{\Frob}{\mathrm{Frob}}
\newcommand{\Gal}{\mathrm{Gal}}

\newcommand{\G}[1]{\mathrm{G}_{#1}}

\newcommand{\GK}{\mathrm{G}_{K}}
\newcommand{\GL}{\mathrm{GL}}
\newcommand{\GQ}{\mathrm{G}_{\Q}}
\newcommand{\gross}{Gr\"o\ss encharakter}

\newcommand{\Ind}{\operatorname{Ind}}

\newcommand{\norm}{\operatorname{Norm}_{K/\Q}}

\newcommand{\ord}{\mathrm{ord}}
\newcommand{\pgcd}{\mathrm{pgcd}}
\newcommand{\PGL}{\mathrm{PGL}}
\newcommand{\place}{v}
\newcommand{\proj}{\mathbf{P}}

\newcommand{\Q}{\mathbf{Q}}
\newcommand{\Qbar}{\overline{\mathbf{Q}}}

\newcommand{\tr}{\mathrm{tr}}

\newcommand{\Z}{\mathbf{Z}}
\newcommand{\Zbar}{\overline{\mathbf{Z}}}
\DeclareMathOperator{\rec}{rec}

\newcommand{\fp}{\mathfrak{p}}
\newcommand{\ppcm}{\mathrm{ppcm}}

\title[Repr\'esentations di\'edrales et formes CM]{Repr\'esentations galoisiennes di\'edrales \\ et formes \`a multiplication complexe}

\author[N.~Billerey]{Nicolas Billerey $^\ast$}
\thanks{$^\ast$ N.B.~remercie le projet ANR-14-CE25-0015 Gardio de l'Agence Nationale de la Recherche et la F\'ed\'eration de Recherche en Math\'ematiques Rh\^one-Alpes-Auvergne (CNRS FR 3490) pour leur soutien financier.}
\address{Université Clermont Auvergne\\
LMBP 
UMR 6620 -- CNRS\\
Campus des Cézeaux\\
3, place Vasarely\\
F-63178 Aubière, France}

\email{Nicolas.Billerey@uca.fr}

\urladdr{http://math.univ-bpclermont.fr/\char'176billerey/}
\author[F.\thinspace{}A.\thinspace{}E. Nuccio]{Filippo\thinspace{}A.\thinspace{}E. Nuccio Mortarino Majno di Capriglio$^\dagger$}
\thanks{$^\dagger$ Lors de la r\'edaction de ce travail, F.N.~a b\'en\'efici\'e d'une d\'echarge de service de l'Universit\'e Jean Monnet de Saint-Étienne, et tient \`a remercier les coll\`egues qui l'ont rendue possible.}
\address{Université de Lyon\\
Institut Camille Jordan\\
UMR 5208 -- CNRS\\
Université Jean Monnet\\
23, rue du Docteur Paul Michelon\\
F-42023 Saint-Étienne, France}

\email{filippo.nuccio@univ-st-etienne.fr}

\urladdr{http://perso.univ-st-etienne.fr/nf51454h/}
\date{\today}


\begin{document}

\begin{abstract}
Pour une repr\'esentation galoisienne di\'edrale en caract\'eristique~${\ell}$ on \'etablit (sous certaines hypoth\`eses) l'existence d'une newform \`a multiplication complexe, dont on contr\^ole le poids, le niveau et le caract\`ere, telle que la repr\'esentation $\ell$-adique associ\'ee est congrue modulo $\ell$ \`a celle de d\'epart.

Given a dihedral Galois representation in characteristic $\ell$, we establish (under some assumption) the existence of a CM newform, whose weight, level and Nebentypus we pin down, such that its $\ell$-adic representation is congruent modulo $\ell$ to the one we started with.
\end{abstract}

\maketitle

\section{Introduction}\label{sec:introduction}

Soit~\(\Qbar\) une cl\^oture alg\'ebrique de~\(\Q\), \(\ell\) un nombre premier et~\(\FFbar\) une cl\^oture alg\'ebrique du corps~\(\F_{\ell}\) \`a~\(\ell\) \'el\'ements. Par repr\'esentation galoisienne, on entend ici un homomorphisme irr\'eductible et continu \(\rho\colon\GQ\rightarrow\GL_2(\FFbar)\) o\`u $\GQ=\Gal(\Qbar/\Q)$. On dit qu'une telle repr\'esentation galoisienne est modulaire si elle est isomorphe \`a la r\'eduction d'une repr\'esentation \(\ell\)-adique associ\'ee \`a une forme parabolique propre~\(f\) de poids \(\geq2\). Dans ce cas, on dit alors aussi que \(\rho\) provient de la forme~\(f\).

On dit qu'une repr\'esentation galoisienne~\(\rho\) est di\'edrale si son image, vue dans \(\PGL_2(\FFbar)\), est isomorphe au groupe di\'e\-dral~\(D_{n}\) d'ordre \(2n\) avec~$n\geq3$. Soit~$C_n$ l'unique sous-groupe cyclique d'ordre~$n$ de~$D_n$. On note alors $K$ le sous-corps quadratique de~$\Qbar$ laiss\'e fixe par le noyau du caract\`ere
\[
\G{\bf{Q}}\stackrel{\proj\rho}{\longrightarrow} D_{n}\longrightarrow D_n/C_n\simeq\{\pm1\}
\]
o\`u $\proj\rho$ est la compos\'ee de $\rho$ avec la projection~$\GL_2(\FFbar)\rightarrow\PGL_2(\FFbar)$.

On rappelle qu'une newform \(g=\sum_{n\geq1}c_nq^n\) est dite \`a multiplication complexe s'il existe un caract\`ere de Dirichlet non trivial~\(\nu\) tel que pour tout~\(p\) dans un ensemble de nombres premiers de densit\'e~\(1\), on a~\(c_p=\nu(p)c_p\). On montre alors que le corps~\(F\) correspondant au noyau de~$\nu$ est quadratique imaginaire et on dit aussi que~\(g\) a multiplication complexe par~\(F\).

Il est bien connu que les repr\'esentations galoisiennes attach\'ees aux formes \`a multiplication complexe sont, en g\'en\'eral, di\'edrales. Par ailleurs, c'est un cas particulier de la c\'el\`ebre conjecture de modularit\'e de Serre (qui \'etait connu longtemps avant la d\'emonstration g\'en\'erale de Khare--Wintenberger) qu'une repr\'esentation di\'edrale impaire provient d'une forme modulaire de poids~$\geq2$. On trouve une preuve moderne de ce résultat dans~\cite{Wie04}, optimale par rapport au poids et au niveau, mais qui ne fournit pas de renseignement sur la nature de la forme modulaire correspondante. La construction d'une telle forme de poids~$\ge2$ est aussi esquissée dans \cite{DeSe74}, en combinant l'exemple p.~517 avec les~\S\S6.9 et~6.10 ; l\`a encore, rien ne justifie qu'elle soit  \`a multiplication complexe.

Dans ce travail, on s'int\'eresse \`a la question plus pr\'ecise de d\'eterminer si une repr\'esentation galoisienne di\'edrale \(\rho\) donn\'ee provient d'une forme modulaire \emph{\`a multiplication complexe de poids~\(k(\rho)\)}, o\`u \(k(\rho)\geq2\) d\'esigne le poids de Serre de la repr\'esentation~\(\rho\) (d\'efini dans~\cite[\S2]{Ser87}). Dans ce cas, on souhaite \'egalement d\'eterminer le niveau minimal de la forme correspondante. 

Le r\'esultat que l'on obtient dans cette direction est le suivant o\`u l'on a not\'e~\(N(\rho)\) la partie premi\`ere \`a~\(\ell\) du conducteur d'Artin de~\(\rho\) (\emph{loc.~cit.},~\no1.2) et $\varepsilon(\rho)$ le caract\`ere associ\'e \`a~$\rho$ par Serre (\emph{loc.~cit.},~\no1.3).
\begin{theoreme}\label{thm:main}
Soit \(\rho\) une repr\'esentation galoisienne di\'edrale. Avec les notations pr\'ec\'edentes, on suppose\textup{:}
\begin{enumerate}[label=\textup{(}\roman*\textup{)}]
	\item \(K\) est quadratique imaginaire\textup{;}
	\item \(2\le k(\rho)\le \ell-1\) et \(\ell\ge5\).
\end{enumerate}
Alors, \(\rho\) est modulaire et provient d'une newform \`a multiplication complexe par le corps~\(K\), de poids~\(k(\rho)\) et de niveau
\begin{equation*}
N'=
\left\{
\begin{array}{ll}
N(\rho) & \text{si $\ell$  est non ramifi\'e dans~$K$ }; \\
\ell^2N(\rho) & \text{si $\ell$ est ramifi\'e dans~$K$}.
\end{array}
\right.
\end{equation*}
De plus, on a les propri\'et\'es suivantes~: 
\begin{enumerate}[label=\textup{(}\alph*\textup{)}]
\item si \(\ell\) est ramifi\'e dans \(K\), alors \(\ell\in\{2k(\rho)-1,\ 2k(\rho)-3\}\);
\item si $\varepsilon(\rho)$ est trivial, alors la forme \`a multiplication complexe peut \^etre choisie de caract\`ere trivial et de niveau divisant~$N'$.
\end{enumerate}
\end{theoreme}
\begin{remarques}\leavevmode
\begin{enumerate}[label=\arabic*.]
\item Dans le cas $\varepsilon(\rho)=1$ et $k(\rho)=2$, le Th\'eor\`eme~\ref{thm:main} est d\'emontr\'e dans~\cite{Nua11}.
\item Le r\'esultat est optimal au sens suivant. D'une part, si \(\rho\) provient d'une newform (\`a multiplication complexe ou non), alors celle-ci est de niveau divisible par~\(N(\rho)\)~: cela r\'esulte d'un th\'eor\`eme de Carayol (voir~\cite[Th\'eor\`eme~(A)]{Car86} et les Sections~1.--2. de~\cite{Car89}). D'autre part, l'exemple~\ref{subsec:Delta} de la Section~\ref{sec:exemples} montre que le niveau propos\'e ne peut, en g\'en\'eral, \^etre abaiss\'e  dans le cas o\`u~$\ell$ est ramifi\'e dans~$K$.
\item L'hypoth\`ese que $K$ soit imaginaire entraîne en particulier que $\rho$ est impaire, dans le sens que le déterminant de $\rho(c)$ vaut $-1$ pour toute conjugaison complexe $c\in\GQ$.
\end{enumerate}
\end{remarques}

Soit~\(A/\Q\) une vari\'et\'e ab\'elienne simple. On note~\(\End_{\Q}(A)\) l'anneau de ses endomorphismes d\'efinis sur~$\Q$. Suivant une terminologie de Ribet, on dit que \(A\) est de type \(\GL_2\) si \(E=\End_{\Q}(A)\otimes\Q\) est un corps de nombres de degr\'e~\(\dim(A)\). Pour toute place finie \(\lambda\) de~\(E\) au-dessus de~\(\ell\) de corps r\'esiduel~\(\F_{\lambda}\), on note alors
\[
\rho_{A,\lambda}\colon\GQ\longrightarrow\GL_2(\F_{\lambda})
\]
la repr\'esentation donnant l'action de~\(\GQ\) sur~\(A[\lambda]\) (voir~\cite[page~7]{Rib92}). Le corollaire suivant au Th\'eor\`eme~\ref{thm:main} ci-dessus g\'en\'eralise~\cite[Theorem~1.6]{Che02} et justifie~\cite[Remark~4.4]{GhPa12}; c'est essentiellement une cons\'equence  de~\cite[Theorem~1]{Nua11}. Pour les d\'efinitions et propri\'et\'es des sous-groupes de Cartan, voir~\cite[\S2]{Ser72} ou~\cite[Chapter~XI, \S2]{Lan76}.

\begin{corollaire}\label{cor}
Soit~\(A/\Q\) une vari\'et\'e ab\'elienne de type \(\GL_2\) de 
conducteur~\(N_A\) et~\(\lambda\) une place finie de \(E=\End_{\Q}(A)\otimes\Q\) 
au-dessus de~\(\ell\). On suppose~\(\ell\geq5\), \(\ell\nmid N_A\) et l'image de 
\(\rho_{A,\lambda}\) contenue dans le normalisateur d'un sous-groupe de Cartan 
non d\'eploy\'e de~\(\GL_2(\F_{\lambda})\). Alors, \(\rho_{A,\lambda}\) provient 
d'une newform \`a multiplication complexe de poids~\(2\) et de 
niveau~\(N(\rho_{A,\lambda})\) \textup{(}divisant~\(N_A\)\textup{)} qui, 
de plus, est de caract\`ere trivial lorsque~$E$ est totalement r\'eel et $A$ a 
tous ses endomorphismes d\'efinis sur~$\Q$.
\end{corollaire}
\begin{proof}
D'apr\`es~\cite[Lemma~3.1]{Rib92}, on a \(\det(\rho_{A,\lambda})=\epsilon\chi\), 
o\`u \(\chi\) d\'esigne le caract\`ere cyclotomique mod~\(\ell\) et \(\epsilon\) 
est un caract\`ere non ramifi\'e hors de~\(N_A\) 
et m\^eme trivial lorsque~$E$ est totalement r\'eel et $A$ a tous ses 
endomorphismes d\'efinis sur~$\Q$ \cite[Lemma~4.5.1]{Rib76}. Soit~\(G_{\ell}\) 
un groupe de d\'ecomposition en~\(\ell\) de~\(\GQ\) et~\(I_{\ell}\) son 
sous-groupe d'inertie. On sait que la semi-simplifi\'ee 
de~$\rho_{A,\lambda}|_{I_{\ell}}\) se factorise par l'inertie mod\'er\'ee et 
qu'elle est donc diagonali\-sable (\cite[{Proposition~4~et s.}]{Ser72}). On note 
$\phi$ et $\phi'$ les deux caract\`eres correspondant. 
D'apr\`es~\cite[Corollaire 3.4.4]{Ray74}, on peut \'ecrire 
$\phi=\psi_2^{a}\psi_2'^{b}$ et $\phi'=\psi_2^{n}\psi_2'^{m}$ o\`u \(\psi_2\) et 
$\psi_2'$ sont les deux caract\`eres fondamentaux de niveau~\(2\) 
(\cite[\no1.7]{Ser72}) et o\`u $a,b,n$ et~$m$ sont des entiers \'egaux \`a $0$ 
ou~$1$. Comme $\phi\phi'=\chi=\psi_2\psi_2'$, il vient~:
\[
\{\phi,\phi'\}=\{\psi_2,\psi_2'\}\quad\text{ou}\quad \{\phi,\phi'\}=\{1,\chi\}.
\]
Le premier cas donne~\(k(\rho_{A,\lambda})=2\) d'apr\`es~(2.8.1) 
de~\cite[Proposition~3]{Ser87}. 
En particulier, on a $\varepsilon(\rho_{A,\lambda})=\epsilon$. Par 
ailleurs, l'image de~\(\rho_{A,\lambda}\) ne contient pas d'\'el\'ement 
d'ordre~\(\ell\). Dans le second cas, on a donc
\[
\rho_{A,\lambda}|_{I_{\ell}}\simeq\begin{pmatrix}
\chi & 0 \\
0 & 1\\
\end{pmatrix}
\]
et on conclut \`a~\(k(\rho_{A,\lambda})=2\) d'apr\`es~(2.8.2) de \emph{loc. cit.} Or, la représentation~$\rho_{A,\lambda}$ est impaire d'apr\`es~\cite[Lemma~3.2]{Rib92}
. Pour toute conjugaison complexe~\(c\in \GQ\), l'\'el\'ement \(\rho_{A,\lambda}(c)\) est donc conjugu\'e \`a la matrice diagonale de valeurs propres~$\{-1,+1\}$. Comme celles-ci ne sont pas conjugu\'ees sur~$\F_{\lambda}$, l'\'el\'ement~\(\rho_{A,\lambda}(c)\) n'est contenu dans aucun sous-groupe de Cartan non d\'eploy\'e de~$\GL_2(\F_{\lambda})$ (\cite[page~181]{Lan76}). En particulier, la repr\'esentation $\rho_{A,\lambda}$ est di\'edrale et le corps quadratique~\(K\) correspondant est imaginaire. Comme~\(\ell\geq5\) et \(k(\rho_{A,\lambda})=2\), on d\'eduit du Th\'eor\`eme~\ref{thm:main} le r\'esultat voulu.
\end{proof}

L'article est organis\'e de la fa\c con suivante. La Section~\ref{sec:TCC} contient des rappels sur les \gross e et un r\'esultat (Proposition~\ref{prop:TCC}) essentiel \`a la d\'emonstration du th\'eor\`eme principal (Th\'eor\`eme~\ref{thm:main}) qui, elle, occupe la Section~\ref{sec:demonstration}. La derni\`ere section est consacr\'ee \`a deux exemples num\'eriques.

\medskip

\emph{Remerciements.-- }Les auteurs remercient Gebhard B\"ockle pour une question ayant men\'e \`a ce travail. N.B.~est \'egalement reconnaissant envers Imin Chen, Luis Dieulefait, Pierre Lezowski et Joan Nualart pour d'int\'eressantes discussions.

\section{Une proposition de la th\'eorie du corps de classes}\label{sec:TCC}

Dans cette section, on rappelle quelques notions sur les \gross e puis on 
d\'emontre une proposition (Proposition~\ref{prop:TCC}) de la th\'eorie du corps 
de classes qui nous sera utile au paragraphe~\ref{subsec:dem_thm_main}. Dans 
toute la suite, $K$ d\'esigne un corps quadratique imaginaire contenu 
dans~$\Qbar$. On identifie~$\Qbar$ \`a un sous-corps de~$\C$ et on fixe, une 
fois pour toutes, un plongement~$\Qbar\hookrightarrow\Qellbar$ o\`u $\Qellbar$ 
d\'esigne une cl\^oture alg\'ebrique de~$\Q_\ell$ de corps r\'esiduel~$\FFbar$. 
Cela induit une place de~$\Qbar$ au-dessus de~$\ell$ not\'ee~$\place$.

On adopte par ailleurs les notations suivantes~:

\(\mathcal{O}_K\) est l'anneau des entiers de~\(K\), \(\mathcal{O}_K^\times\) 
son groupe des unit\'es et $-D_K$ son discriminant;

\(\Sigma_K\) est l'ensemble des places ultram\'etriques de~\(K\);

\(K_w\) est  le compl\'et\'e de~\(K\) en~\(w\) (\(w\in\Sigma_K\));

\(\mathcal{O}_w^{\times}\), \(\pi_w\), \(p_w\) sont respectivement le groupe des unit\'es, une uniformisante et la carac\-t\'eristique r\'esiduelle du corps local~\(K_w\) (\(w\in\Sigma_K\));

\(\mathcal{O}_w^{(m)}\) est l'ensemble des unit\'es \(u\in\mathcal{O}_w^{\times}\) congrues \`a $1$ modulo~$\pi_w^m$ avec $m\geq1$ entier (\(w\in\Sigma_K\));

\(\mathfrak{p}_w\) est l'id\'eal premier de~\(\mathcal{O}_K\) induit par~\(w\) (\(w\in\Sigma_K\));

\(\AK\) est le groupe des id\`eles de $K$; si \(a\in\AK\) et \(w\in\Sigma_K\), on note~\(a_w\) la composante de~\(a\) en~\(w\);

$C_K=\AK/K^\times$ est  le groupe des classes d'id\`eles de~\(K\);

$\Afin$ (resp. $\Ainf$) est l'ensemble des id\`eles finis (resp. infinis) de~$K$.

\'Etant donn\'es une place~$w\in\Sigma_K$ et un id\'eal fractionnaire~$\mathfrak{m}$ de~$K$, on note~$\ord_w(\mathfrak{m})$ la valuation de~$\mathfrak{m}$ en l'id\'eal premier~$\mathfrak{p}_w$. On d\'efinit
\[
U_{\mathfrak{m}}=\left\{a\in\AK\, \big|\, w(a_w-1)\geq\ord_w(\mathfrak{m}),\ \forall{w\in\Sigma_K}\right\}
\]
et $E_{\mathfrak{m}}=U_{\mathfrak{m}}\cap\mathcal{O}_K^{\times}$.

Enfin, si $\mathfrak{p}$ est un id\'eal premier de~$\mathcal{O}_K$, on note $\pi_{\mathfrak{p}}$ une uniformisante locale en la place de~$K$ induite par~$\mathfrak{p}$.

\subsection{}\label{ss:car_S_m}  Par \gross~\(\chi\) de~\(K\) on entend ici un homomorphisme de groupes continu
\[
\chi\colon\AK\longrightarrow\C^{\times}
\]
tel que~\(\chi(K^{\times})=1\). Si $w\in\Sigma_K$, on note $\chi_w$ la composante en~$w$ de~$\chi$ et on d\'esigne par $\chi_{\mathrm{f}}$ (resp.~$\chi_{\infty}$) la partie finie (resp.~infinie) de~$\chi$. On dit que $\chi$ est ramifi\'e en $w\in\Sigma_K$ s'il existe une unit\'e~$u\in\mathcal{O}_w^{\times}$ telle que~$\chi_w(u)\not=1$. Dans le cas contraire, on dit que~\(\chi\) est non ramifi\'e en~$w$. Par continuit\'e, $\chi$ est non ramifi\'e en toutes les places de~$K$ sauf un nombre fini. Si $\chi$ est ramifi\'e en~$w\in\Sigma_K$, il existe un entier $m_w\geq1$ minimal tel que $\chi_w\left(\mathcal{O}_w^{(m_w)}\right)=1$. Le conducteur de~$\chi$ est alors, par d\'efinition, l'id\'eal de~$\mathcal{O}_K$ 
\[
\mathfrak{f}_\chi=\prod_{w}\mathfrak{p}_w^{m_w}
\]
o\`u $w\in\Sigma_K$ parcourt l'ensemble des places o\`u $\chi$ est ramifi\'e.

On dit enfin qu'un \gross\ \(\chi\) est de type \`a l'infini~$(m,n)$ avec $m,n\in\Z$ lorsque 
\[
\chi_\infty(z)=\frac{1}{z^{m}(z^c)^{n}},\quad \text{pour tout }z\in\Ainf\simeq\C^\times,
\]
o\`u $z^c$ d\'esigne le conjugu\'e complexe de~$z$. Un tel \gross\ est de type~(\(A_0\)) dans la terminologie de Weil (\cite[page~4]{Wei56}).

\subsection{}\label{subsec:teichmuller} On note~\(z\mapsto\overline{z}\) l'homomorphisme de r\'eduction~\(\Zellbar\rightarrow\FFbar\) o\`u~$\Zellbar$ d\'esigne l'anneau des entiers de~$\Qellbar$.

Soit $\Zbar$ l'anneau des entiers de~$\Qbar$. Il existe alors (\cite[Chapter~2]{Was97}) un unique homomorphisme de groupes~\(\FFstar\rightarrow\Zbar^{\times}\) not\'e~\(x\mapsto\widetilde{x}\) \`a valeurs dans les racines de l'unit\'e d'ordre premier \`a~\(\ell\) tel que 
\[
x=\overline{\widetilde{x}},\quad\text{pour tout }x\in\FFstar.
\]
Par ailleurs, si~\(\zeta\) est une racine de l'unit\'e d'ordre premier \`a~\(\ell\), on a~\(\zeta=\widetilde{\overline{\zeta}}\).

\'Etant donn\'e un groupe~\(G\) et un caract\`ere~\(\varphi\colon G\to\FFstar\), on d\'efinit le rel\`evement de Teichmüller (ou multiplicatif)$$
\widetilde{\varphi}\colon G\longrightarrow\Qbar^\times
$$ 
de~\(\varphi\) comme le compos\'e de~\(\varphi\) avec le morphisme~\(x\mapsto\widetilde{x}\) ci-dessus. C'est le seul caract\`ere \`a valeurs dans les racines de l'unit\'e d'ordre premier \`a~\(\ell\) de~\(\Qbar\) v\'erifiant
\[
\overline{\widetilde{\varphi}(x)}=\varphi(x),\quad\text{pour tout }x\in G.
\]

\subsection{} \'Etant donn\'e un \gross\ $\chi$ de~$K$ de type \`a l'infini~$(m,n)$ avec $m,n\geq0$, on  v\'erifie que pour toute place~$w\in\Sigma_K$, on a~$\chi(\pi_w)\in\Zbar$, o\`u l'on note encore $\pi_w\in\Afin$  l'id\`ele ayant composantes triviales en dehors de $w$ et qui coïncide avec  l'uniformisante~$\pi_w$ choisie en~$w$. En effet, si $h$ d\'esigne le nombre de classes de~$K$, l'id\'eal $\mathfrak{p}_w^h$ est principal et engendr\'e par, disons, $a\in\mathcal{O}_K$. Soit $x$ l'id\`ele $\pi_w^ha^{-1}$. Alors~$x$ est une unit\'e en toute place finie de~$K$ et on a
\[
\chi(\pi_w)^h=\chi\left(\pi_w^h\right)\chi\left(a^{-1}\right)=\chi(x)=\chi_{\mathrm{f}}(x)a^m (a^c)^n\in\Zbar.
\]
La proposition suivante est le r\'esultat principal de cette section; elle g\'en\'eralise notamment~\cite[Proposition~3.1]{Che02}.
\begin{proposition}\label{prop:TCC} Supposons~$\ell\geq5$. Soit $\alpha$ un \gross\ de~$K$ d'image finie et de conducteur $\mathfrak{f}_\alpha$, et soit $k$ un entier~$\ge2$. Alors, il existe un \gross\ $\delta$ de $K$ de type \`a l'infini~$(k-1,0)$ et conducteur~\(\mathfrak{f}_{\delta}\)  v\'erifiant
\begin{equation}\label{eq:ordw(delta)}
\ord_w(\mathfrak{f}_\delta)=\ord_w(\mathfrak{f}_\alpha)\quad \text{pour toute place }w\in\Sigma_K\setminus\{v\}\tag{\ref{prop:TCC}--1}
\end{equation}
et tel quel pour toute place~$w\in\Sigma_K\setminus\{v\}$ premi\`ere \`a~$\mathfrak{f}_\alpha$, on a 
\begin{equation}\label{eq:congruences}
\overline{\delta(\pi_w)}=\overline{\alpha(\pi_w)}.\tag{\ref{prop:TCC}--2}
\end{equation}
De plus, on a 
\begin{enumerate}[label=\textup{(}\alph*\textup{)}]
\item\label{condition:falpha>1} si $\ord_{\place}(\mathfrak{f}_\alpha)\geq 2$, alors $\ord_{\place}(\mathfrak{f}_\delta)=\ord_{\place}(\mathfrak{f}_\alpha)$\textup{;}
\item\label{condition:falpha=1} si $\ord_{\place}(\mathfrak{f}_\alpha)=1$, alors
\[
\ord_{\place}(\mathfrak{f}_\delta)=
\left\{
\begin{array}{ll}
0 & \text{si pour toute unit\'e $u\in\mathcal{O}_{\place}^\times$ on a $\alpha_{\place}(u)=\widetilde{\overline{u}}^{1-k}$} \\
1 & \text{sinon}
\end{array}
\right.;
\]
\item\label{condition:falpha<1} si $\ord_{\place}(\mathfrak{f}_\alpha)=0$, alors 
\[
\ord_{\place}(\mathfrak{f}_\delta)=
\left\{
\begin{array}{ll}
0 & \text{si $\ell^f-1\mid k-1$} \\
1 & \text{sinon}
\end{array}
\right.
\] 
o\`u $f$ est le degr\'e r\'esiduel de $K$ en~\(\ell\).
\end{enumerate}
\end{proposition}
\begin{proof}
Comme on a suppos\'e $K$ quadratique imaginaire et $\ell\ge5$, on a, avec les notations pr\'ec\'edentes, $E_{\mathfrak{p}_\place}=U_{\mathfrak{p}_\place}\cap\mathcal{O}_K^\times=\{1\}$. 
D'apr\`es~\cite[page~286]{Ser72}, il existe alors \(f\in \operatorname{Hom}(\AK,\C^\times)\) tel que
\begin{enumerate}[label=\Roman*.]
\item \(f(x)=1\) pour tout \(x\in U_{\mathfrak{p}_\place}\),
\item\label{item:car_S_m_cond_ii} \(f(x)=x^{k-1}\) pour tout \(x\in K^{\times}\).
\end{enumerate}
Posons alors 
\[
\beta(x)=f(x)x_{\infty}^{1-k},\quad\text{pour tout~\(x\in\AK\)},
\]
o\`u~\(x_{\infty}\in\Ainf\) d\'esigne la composante \`a l'infini de l'id\`ele~\(x\).  Comme $U_{\mathfrak{p}_\place}$ est ouvert dans~$\AK$ et contenu dans le noyau de $f$, il s'en suit que~$f$, puis l'homomorphisme~\(\beta\), sont continus. Par ailleurs, si~\(x\in K^{\times}\), on a d'apr\`es la propri\'et\'e~\ref{item:car_S_m_cond_ii} ci-dessus
\[
\beta(x)= f(x) x^{1-k} =1.
\]
Autrement dit, \(\beta\) est un \gross\ de~\(K\). De plus, si~\(w\in\Sigma_K\) et~\(x\in\mathcal{O}_w^{\times}\cap U_{\mathfrak{p}_\place}\), on a~\(\beta_w(x)=1\). En particulier, \(\beta\) est de conducteur divisant~\(\mathfrak{p}_\place\). Enfin, si $x\in\Ainf$, on a $x\in U_{\mathfrak{p}_\place}$, puis $\beta_\infty(x)=x^{1-k}$. Le \gross\ $\beta$ est donc de type \`a l'infini~$(k-1,0)$.

Comme l'a montr\'e Weil dans \cite[page~5]{Wei56}, il existe une extension finie~$L$ de~$\Q$, qu'on regarde plong\'ee dans $\Qbar$, qui contient les valeurs de $\beta_w$ pour toute place~$w$. Quitte \`a l'\'elargir, on peut de plus supposer qu'elle contient~$K$. On note~$L_\place$ le compl\'et\'e de~$L$ en la place de~$L$ induite par la place~$\place$ de~$\Qbar$ et on commence par d\'efinir un caract\`ere $\gamma_0\colon \AK\rightarrow L_{\place}^\times$ par la r\`egle
\[
\left\{\begin{array}{lclr}
(\gamma_0)_w & =&\beta_w^{-1}& \quad (w\in\Sigma_K\setminus\{v\})\\
(\gamma_0)_{\place}& =&\beta_{\place}^{-1}\cdot(-)^{k-1}&\\
(\gamma_0)_\infty& =& 1 & 
\end{array}\right.,
\]
o\`u $(-)^{k-1}$ d\'esigne l'\'el\'evation \`a la puissance~$k-1$. En tant que produit de caract\`eres continus, $\gamma_0$ est continu. Soit $a\in K^\times$. On a alors
\begin{align*}
\gamma_0(a) & =\Big(\prod_{w\in\Sigma_K\setminus\{\place\}}\beta_w(a)^{-1}\Big)\cdot\beta_{\place}(a)^{-1}\cdot a^{k-1}\\
 & =\beta_\mathrm{f}(a)^{-1}\cdot a^{k-1}=\beta_\infty(a)\cdot a^{k-1}=1
\end{align*}
car $\beta$ est trivial sur les id\`eles principaux et de type \`a l'infini~$(k-1,0)$. Il s'en suit que l'on peut regarder $\gamma_0$  en tant que caract\`ere continu de \(C_K\) et aussi du quotient~\(C_K/\Aconc{\infty}\). Comme $K$ est totalement imaginaire, l'isomorphisme de r\'eciprocit\'e de la th\'eorie du corps de classes globale identifie ce quotient avec l'ab\'elianis\'e $\GK^\text{ab}$ de~$\GK$, et montre en particulier qu'il s'agit d'un groupe compact. L'image de l'homomorphisme continu~$\gamma_0$ est donc contenue dans le groupe~$\mathcal{O}_{L_\place}^\times$ des unit\'es de~$L_\place$. On note
\[
\overline{\gamma_0}\colon \AK\longrightarrow \FFstar
\]
la compos\'ee de~$\gamma_0$ avec l'homomorphisme de r\'eduction $\Zellbar\rightarrow\FFbar$. C'est un caract\`ere d'image finie. On d\'efinit alors
\[
\gamma=\widetilde{\overline{\gamma_0}}\colon \AK\longrightarrow \Zbar^\times\subset\C^\times
\]
comme \'etant le rel\`evement de Teichmüller de~$\overline{\gamma_0}$ 
d\'etaill\'e au~paragraphe~\ref{subsec:teichmuller}. Il s'agit donc d'un 
\gross\ de type \`a l'infini~$(0,0)$, type dit \og trivial\fg.

Pour compl\'eter la preuve de la proposition, posons $\delta=\alpha\beta\gamma$: il s'agit bien d'un \gross\ de type \`a l'infini \'egal au type \`a l'infini de $\beta$, \`a savoir $(k-1,0)$, puisque tant $\alpha$ que $\gamma$ ont un type \`a l'infini trivial. Il reste \`a pr\'esent \`a d\'eterminer~$\mathfrak{f}_\delta$ et \`a d\'emontrer les congruences~\eqref{eq:congruences}.

Montrons tout d'abord les congruences~\eqref{eq:congruences}. Soit~\(w\in\Sigma_K\setminus\{v\}\) une place ultram\'etrique de~\(K\) en laquelle~$\alpha$ est non ramifi\'e. D'apr\`es ce qui pr\'ec\`ede, on a~\(\gamma_0(\pi_w)=\beta(\pi_w)^{-1}\in\Zellbarstar\) et, par ailleurs,~\(\alpha(\pi_w)\in\Zellbar\). Par r\'eduction, il vient donc
\begin{equation*}
\overline{\delta(\pi_w)}=\overline{\alpha(\pi_w)}\, \overline{\beta(\pi_w)}\, \overline{\beta(\pi_w)^{-1}}
=\overline{\alpha(\pi_w)}
\end{equation*}
et le r\'esultat souhait\'e.

\'Etudions \`a pr\'esent le conducteur~$\cond_\delta$. Soit $w$ une place finie de~$K$ distincte de~$\place$ et $u\in\mathcal{O}_w^\times$. Comme $\beta$ est non ramifi\'e en~$w$, on a $\beta_w(u)=1=\gamma_w(u)$ et donc \(\delta_w(u)=\alpha_w(u)\), ce qui entra\^ine~\eqref{eq:ordw(delta)}.

Passons maintenant \`a la d\'emonstration des points~\ref{condition:falpha>1}--\ref{condition:falpha<1}. Soit~$u\in\mathcal{O}_\place^\times$. Comme~$\mathfrak{f_\beta}\mid\mathfrak{p}_v$, le caract\`ere $\beta_\place$ est trivial sur~$\mathcal{O}_\place^{(1)}$ et \(\beta_\place(u)\) est une racine de l'unit\'e d'ordre premier \`a~\(\ell\). On a donc~\(\gamma_\place(u)=\beta_\place(u)^{-1}\widetilde{\overline{u}}^{k-1}\), puis
\begin{equation}\label{eq:delta_et_alpha_place}
\delta_\place(u)=\alpha_\place(u)\widetilde{\overline{u}}^{k-1}.
\end{equation}
En particulier, on a \(\delta_\place(u)=\alpha_\place(u)\), pour toute unit\'e $u\in\mathcal{O}_\place^{(1)}$.

\begin{enumerate}[label=(\alph*)]
\item Supposons $\ord_{\place}(\mathfrak{f}_\alpha)\geq 2$. Alors, il existe une unit\'e $u\in\mathcal{O}_\place^{(1)}$ telle que $\delta_\place(u)=\alpha_\place(u)\neq 1$. D'o\`u $\ord_\place(\mathfrak{f}_\delta)\geq2$ et \(\ord_\place(\mathfrak{f}_\delta)=\ord_\place(\mathfrak{f}_\alpha)$ car les restrictions de~$\delta_\place$ et~$\alpha_\place$ \`a~$\mathcal{O}_\place^{(1)}$ co\"incident.
\item Supposons $\ord_{\place}(\mathfrak{f}_\alpha)=1$. Pour toute unit\'e $u\in\mathcal{O}_\place^{(1)}$, on a alors $\delta_\place(u)=1$, d'o\`u $\ord_{\place}(\mathfrak{f}_\delta)\leq 1$. Par ailleurs, d'apr\`es~\eqref{eq:delta_et_alpha_place}, on a $\ord_\place(\mathfrak{\mathfrak{f}_\delta})=1$ si, et seulement si, il existe une unit\'e~$u\in\mathcal{O}_\place^\times$ telle que $\alpha_\place(u)\not=\widetilde{\overline{u}}^{1-k}$. 
\item Supposons $\ord_{\place}(\mathfrak{f}_\alpha)=0$. Alors pour toute unit\'e $u\in\mathcal{O}_\place^\times$, on a \(\delta_\place(u)=\widetilde{\overline{u}}^{k-1}\) et donc le caract\`ere~$\delta_\place$ est trivial sur~$\mathcal{O}_\place^{(1)}$, c'est-\`a-dire $\ord_{\place}(\mathfrak{f}_\delta)\leq 1$. Par ailleurs, le quotient $\mathcal{O}_\place^\times/\mathcal{O}_\place^{(1)}$ s'identifie \`a~$\left(\mathcal{O}_K/\mathfrak{p}_\place\right)^\times$. En particulier, on a $\ord_\place(\mathfrak{\mathfrak{f}_\delta})=0$ si, et seulement si, $\ell^f-1$ divise~$k-1$.
\end{enumerate}
Cela termine la d\'emonstration de la Proposition~\ref{prop:TCC}.
\end{proof}

\section{D\'emonstration du th\'eor\`eme principal}\label{sec:demonstration}

\subsection{\'Etude locale}\label{subsec:descr_loc} On commence par rappeler la d\'efinition centrale suivante.
\begin{definition}
Une repr\'esentation galoisienne (i.e.~un homomorphisme de groupes irr\'eductible et continu) $\rho\colon\GQ\rightarrow\GL_2(\FFbar)$ est dite di\'edrale si son image, vue dans \(\PGL_2(\FFbar)\), est isomorphe au groupe di\'e\-dral~\(D_{n}\) d'ordre \(2n\) avec~$n\geq3$.
\end{definition}
Soit \(\rho\colon\GQ\rightarrow\GL_2(\FFbar)\) une repr\'esentation galoisienne satisfaisant aux hypoth\`eses du Th\'eor\`eme~\ref{thm:main}. On v\'erifie \`a l'aide de la classification des sous-groupes finis de~$\GL_2(\FFbar)$ (voir par exemple \cite[Theorem~3.4]{Blo67}) que l'ordre de l'image de~$\rho$ est n\'ecessairement premier \`a~$\ell$. Avec les notations de l'Introduction, on a donc en particulier $\pgcd(n,\ell)=1$. Par construction, $\proj\rho(\GK)=C_n$ est cyclique, avec $\GK=\Gal(\Qbar/K)$. Comme~$\proj\rho(\GK)$ s'obtient \`a partir de~$\rho(\GK)$ apr\`es passage au quotient par des \'el\'ements de son centre, on en d\'eduit que~$\rho(\GK)$ est ab\'elienne. Ainsi, il existe deux caract\`eres 
\begin{equation*}
\blabla,\blabla'\colon\GK\longrightarrow\FFstar
\end{equation*}
tels que $\blabla\not=\blabla'$ et~\(\rho|_{\G{K}}\simeq\blabla\oplus\blabla'\). Soit~\(\sigma\in\G{\bf{Q}}\setminus\G{K}\). On d\'efinit $\widehat{\blabla}\colon\GK\rightarrow\FFstar$ par~\(\widehat{\blabla}(\tau)=\blabla(\sigma^{-1}\tau\sigma)\) pour tout \(\tau\in\G{K}\)~: cette d\'efinition est ind\'ependante de~$\sigma$ car l'image de $\blabla$ est ab\'elienne. Soit $v$ un vecteur propre pour~$\rho|_{\G{K}}$ de valeur propre~$\blabla(\tau)$ pour tout~$\tau\in\GK$. On a alors, pour tout~$\tau\in\GK$,
\[
\rho(\tau)\rho(\sigma)v=\rho(\sigma\sigma^{-1}\tau\sigma)v=\rho(\sigma)\rho(\sigma^{-1}\tau\sigma)v=\widehat{\blabla}(\tau)\rho(\sigma)v
\]
d'o\`u l'on d\'eduit~\(\blabla'=\widehat{\blabla}\). On a donc~$\rho\simeq\Ind^{\GQ}_{\GK}(\blabla)\simeq\Ind^{\GQ}_{\GK}(\widehat{\blabla})$. Dans la base~$\{v,\rho(\sigma)v\}$, la matrice de~$\rho(\sigma)$ s'\'ecrit
\begin{equation}\label{eq:rho_sigma}
\begin{pmatrix}
0 & \blabla(\sigma^2) \\
1 & 0
\end{pmatrix}.
\end{equation}
L'objectif de ce paragraphe est d'\'etudier la ramification de $\blabla$ et de~$K$ en~$\ell$. Soit~\(\G{\ell}\) le groupe de d\'ecomposition de \(\G{\bf{Q}}\) relatif au plongement fix\'e~$\Qbar\hookrightarrow\Qellbar$ et \(I_{\ell}\) son sous-groupe d'inertie. On note pour simplifier \(k=k(\rho)\), avec $k(\rho)$ comme dans~\cite[\S2]{Ser87}. On rappelle que l'on a suppos\'e \(2\le k\le \ell-1\) et \(\ell\ge5\).

L'image de~$\rho$ est d'ordre premier \`a~$\ell$. En particulier, $\rho|_{I_{\ell}}\) se factorise par l'inertie mod\'er\'ee et est diagonali\-sable (\cite[Proposition~4~et s.]{Ser72}). On note $\blaloc$ et $\blaloc'$ les deux caract\`eres correspondants. Ils sont de niveau~$1$ ou~$2$ et lorsqu'ils sont de niveau~$2$, on a $\blaloc'=\blaloc^\ell$ (\cite[Proposition~1]{Ser87}). On rappelle que~$\mathfrak{p}_\place$ d\'esigne l'id\'eal premier de~$K$ au-dessus de~$\ell$ induit par la place~$\place$ (c'est-\`a-dire par le plongement $\Qbar\hookrightarrow\Qellbar$). On consid\`ere alors les groupes de d\'ecomposition \(\G{\mathfrak{p}_\place}\le \G{\ell}\) et d'inertie \(I_{\mathfrak{p}_\place}\le I_{\ell}\).

\begin{enumerate}[label=(\alph*)]
	\item \label{liste:niveau_1}On suppose que $\blaloc$ et $\blaloc'$ sont 
de niveau \(1\). Comme l'image de \(\rho\) ne contient pas d'\'el\'ement 
d'ordre~\(\ell\), l'image par~$\rho$ de l'inertie sauvage est triviale. Comme on 
a suppos\'e \(2\le k\le \ell-1\), on a, d'apr\`es~\cite[(2.3.2)]{Ser87},
	\[
	\rho|_{I_{\ell}}=\begin{pmatrix}
	1 & 0 \\
	0 & \chi^{k-1} \\
	\end{pmatrix}
	\]
	o\`u \(\chi\) d\'esigne le caract\`ere cyclotomique mod~\(\ell\).
	\begin{enumerate}[label=\ref{liste:niveau_1}-\arabic*.]

		\item Supposons \(\ell\) non ramifi\'e dans \(K\). Alors, \(I_{\mathfrak{p}_\place}\) s'identifie \`a \(I_{\ell}\) et, comme \(\ell>k\), le caract\`ere \(\chi^{k-1}\) est non trivial. On en d\'eduit que soit~\(\blabla\) est  non ramifi\'e en \(\sigma(\mathfrak{p}_\place)\) et que sa restriction au sous-groupe d'iner\-tie en \(\mathfrak{p}_\place\) est donn\'ee par la puissance \((k-1)\)-i\`eme du caract\`ere cyclotomique, ou qu'il en est ainsi pour $\blabla'$. Si \(\sigma(\mathfrak{p}_\place)=\mathfrak{p}_\place\), c'est une contradiction dans les deux cas. On a donc montr\'e que \(\ell\) est d\'ecompos\'e dans~\(K\). Par ailleurs, quitte \`a remplacer~\(v\) par~\(\sigma v\), on peut supposer~\(\blabla\) non ramifi\'e en~\(\sigma(\mathfrak{p}_\place)\) et mod\'er\'ement ramifi\'e en~\(\mathfrak{p}_\place\) avec, pour tout \(\tau\in I_{\mathfrak{p}_\place}\), \(\blabla(\tau)=\chi^{k-1}(\tau)\).
		
		\item Supposons \(\ell\) ramifi\'e dans \(K\). Alors, d'apr\`es la d\'efinition (donn\'ee en Introduction) du corps~$K$, on a~\(\proj\rho(I_{\ell})\not\subset C_n\). Or, l'image du caract\`ere $\chi^{k-1}$ est d'ordre $(\ell-1)/\pgcd(\ell-1,k-1)$ et, en vue de la description de $\rho\vert_{I_\ell}$, elle coïncide avec l'image de \(\proj\rho(I_{\ell})\), qui est donc cyclique et par suite d'ordre~$2$. On en tire \(\ell-1=2\cdot\pgcd(\ell-1,k-1)\), d'o\`u~\(\ell=2k-1\) car~$\ell>k$. En outre, le fait que $\proj\rho(I_\ell)$ soit d'ordre $2$ montre que l'indice de ramification de~$\ell$ dans l'extension $\Qbar^{\operatorname{Ker}(\proj\rho)}/\Q$ est \'egal \`a $2$ et par suite $\proj\rho(I_{\mathfrak{p}_\place})=1$. Il vient donc que \(\blabla\) est non ramifi\'e en \(\mathfrak{p}_\place\).
\end{enumerate}
	\item On suppose que $\blaloc$ et $\blaloc'$ sont de niveau \(2\).  Alors, d'apr\`es~\cite[\no2.2]{Ser87}, \(\rho|_{I_{\ell}}\) est irr\'eductible. En particulier, cela exclut d'avoir \(\G{\ell}\subset\G{K}\), i.e.~\(\ell\) d\'ecompos\'e dans \(K\). Autrement dit, \(\ell\) est inerte ou ramifi\'e dans \(K\) et comme on a suppos\'e \(2\le k\le \ell-1\), on a 
	\[
	\rho|_{I_{\ell}}=
	\begin{pmatrix}
	\psi_2^{k-1} & 0 \\
	0 & \psi_2^{\ell(k-1)} \\
	\end{pmatrix},
	\]
	o\`u $\psi_2$ est un caract\`ere fondamental de niveau \(2\) (pour sa d\'efinition, voir \cite[\no1.7]{Ser72}). On trouve que le caract\`ere \(\psi_2^{k-1}\) est d'ordre~
	\[\frac{\ell^2-1}{\pgcd(\ell^2-1,k-1)}>2\]
	et n'est donc trivial sur aucun sous-groupe de~\(I_{\ell}\) d'indice~\(\le2\).  Par cons\'equent, tant \(\blabla\) que~\(\blabla'\) sont ramifi\'es en~\(\mathfrak{p}_\place\) et, quitte \`a remplacer $\blabla$ avec $\blabla'$, on a que~\(\blabla(\tau)=\psi_2^{k-1}(\tau)\) pour tout~\(\tau\in I_{\mathfrak{p}_\place}\).
	
	Supposons enfin \(\ell\) ramifi\'e dans \(K\). Comme pr\'ec\'edemment, l'image projective $\proj\rho(I_{\ell})$ n'est pas contenue dans~$C_n$. Or, \(\proj\rho(I_{\ell})\) est cyclique et s'identifie \`a l'image de \(\psi_2^{(\ell-1)(k-1)}\). On en d\'eduit que \(\proj\rho(I_{\ell})\) est d'ordre~\(2\), puis que \(2\cdot\pgcd\bigl(\ell^2-1,(\ell-1)(k-1)\bigr)=\ell^2-1\), et finalement \(\ell=2k-3\).
\end{enumerate}

On r\'esume les r\'esultats ci-dessus dans la proposition suivante.
\begin{proposition}\label{prop:ramification}
Avec les hypoth\`eses et notations de ce paragraphe, on peut supposer que l'on est dans l'une des situations suivantes.
\begin{enumerate}[label=\textup{(}\alph*\textup{)}]
\item\label{item:prop_ramification_niv1} Supposons $\blaloc$ et $\blaloc'$ de niveau \(1\). Alors, 
	\begin{enumerate}[label=\ref{item:prop_ramification_niv1}-\arabic*.]
	\item\label{item:niveau_1_l_dec} soit \(\ell\) est d\'ecompos\'e dans \(K\) et dans ce cas \(\blabla\) est non ramifi\'e en \(\sigma(\mathfrak{p}_\place)\), mod\'er\'ement ramifi\'e en \(\mathfrak{p}_\place\) et pour tout \(\tau\in I_{\mathfrak{p}_\place}\), on a \(\blabla(\tau)=\chi^{k-1}(\tau)\), o\`u~$\chi$ d\'esigne le caract\`ere cyclotomique mod~\(\ell\).
	\item soit \(\ell\) est ramifi\'e dans \(K\) et dans ce cas on a \(\ell=2k-1\) et \(\blabla\) est  non ramifi\'e en \(\mathfrak{p}_\place\).
	\end{enumerate}
	\item\label{item:prop_ramification_niv2} Supposons $\blaloc$ et $\blaloc'$ de niveau \(2\). Alors, 
\begin{enumerate}[label=\ref{item:prop_ramification_niv2}-\arabic*.]
\item soit \(\ell\) est inerte dans~\(K\);
\item soit \(\ell\) est ramifi\'e dans \(K\) et \(\ell=2k-3\). 
\end{enumerate}
Dans les deux cas, \(\blabla\) est ramifi\'e en \(\mathfrak{p}_\place\) et pour tout \(\tau\in I_{\mathfrak{p}_\place}\), on a
\(\blabla(\tau)=\psi_2^{k-1}(\tau)\), avec \(\psi_2\) caract\`ere fondamental de niveau~\(2\). 
\end{enumerate}
\end{proposition}

\subsection{} On \'etudie ici la variation du conducteur d'Artin par rapport \`a l'op\'eration d'induction. On rappelle que l'on a 
$\rho=\Ind^{\GQ}_{\GK}(\blabla)$. On d\'esigne par~\(\mathfrak{N}(\rho)\) (resp.~$\mathfrak{N}(\blabla)$) le conducteur d'Artin de~\(\rho\) (resp.~de \(\blabla\)) d\'efini dans~\cite[Chapitre~VI, \S3]{Ser68}. Ce sont, par d\'efinition, des id\'eaux de $\Z$ et de $\mathcal{O}_K$, respectivement. On rappelle que, suivant la notation de Serre~\cite[\no1.2]{Ser87}, $N(\rho)$ d\'esigne le g\'en\'erateur positif de la partie premi\`ere \`a~\(\ell\) de~$\mathfrak{N}(\rho)$.

Pour un entier naturel~\(n\), on note \(n^{(\ell)}\) sa partie premi\`ere \`a~\(\ell\). Le r\'esultat suivant r\'esulte de~\cite[Corollary~1]{Tag02}.
\begin{lemme}[Taguchi]\label{lem:cond_Artin} Avec les notations et hypoth\`eses pr\'ec\'edentes, on a
\[
N(\rho)=D_K^{(\ell)}\norm\bigl(\mathfrak{N}(\blabla)\bigr)^{(\ell)}.
\]
\end{lemme}
\begin{remarque} Dans \cite{Tag02} aucune preuve du Corollary 1 n'est offerte et l'auteur fait r\'ef\'erence \`a \cite[Chapter~VI, \S3]{Ser68}, o\`u le proc\'ed\'e de \og globalisation\fg\ est trait\'e dans le cas de repr\'esentations \`a coefficients complexes, et cela en utilisant la th\'eorie des caract\`eres. Il est bien connu (voir, par exemple, \cite[\S 15.5]{Ser78-2}) que cette th\'eorie s'\'etend aux repr\'esentations en caract\'eristique $\ell$ d'un groupe d'ordre premier \`a $\ell$, ce qui est le contexte dans lequel nous travaillons \`a pr\'esent. Les arguments de \cite[Chapitre~VI, \S3]{Ser68} s'adaptent alors pour prouver la validit\'e du r\'esultat cit\'e de Taguchi.
\end{remarque}

\subsection{}\label{ss:calculTCCL} Dans ce paragraphe, on fixe une extension finie $L/\Q_{\ell}$. On note \(k_L\) son corps r\'esiduel de sorte que l'on a \(k_L=\F_q\) o\`u \(q=\left|k_L\right|\) d\'esigne le cardinal de~\(k_L\). On note \(L^{\mathrm{ab}}\) (resp. $L^\mathrm{nr}$) l'extension ab\'elienne (resp. non ramifi\'ee) maximale de \(L\) contenue dans $\Qellbar$. Le morphisme \(\theta_{q-1}\colon \operatorname{Gal}(\Qellbar/L^\mathrm{nr})\to k_L^\times\) construit par Serre dans \cite[\S 1.3]{Ser72} (et dont les conjugu\'es sur~$\F_q$ forment l'ensemble des caract\`eres fondamentaux de niveau~$r$ de $\Qellbar/L$, o\`u~$q=\ell^r$) se factorise par le groupe \(I_t(L^{\mathrm{ab}}/L)\) d'inertie mod\'er\'ee de l'extension \(L^{\mathrm{ab}}/L\) et on note encore 
\[
\theta_{q-1}\colon I_t(L^{\mathrm{ab}}/L)\longrightarrow k_L^{\times}
\]
le morphisme pass\'e au quotient. Par ailleurs, l'application de r\'eciprocit\'e de la th\'eorie du corps de classes locale fournit un homomorphisme surjectif
\[
\omega\colon k_L^{\times}\longrightarrow I_t(L^{\mathrm{ab}}/L).
\]
Un calcul, d\'etaill\'e dans \cite[Proposition 3]{Ser72}, montre alors que l'on a 
\begin{equation}\label{eq:calculTCCL}
\theta_{q-1}(\omega(x))=x^{-1},
\end{equation}
pour tout \(x\in k_L^{\times}\). Ce r\'esultat nous sera utile dans la d\'emonstration qui suit.

\subsection{D\'emonstration du th\'eor\`eme principal}\label{subsec:dem_thm_main}
On reprend les notations et hypoth\`eses du paragraphe~\ref{subsec:descr_loc}. Soit~$\widetilde{\blabla}\colon\GK\rightarrow\Zbar^\times\subset\C^\times$ le rel\`evement multiplicatif de~$\blabla$ d\'efini au paragraphe~\ref{subsec:teichmuller}. Il se factorise par l'ab\'elianis\'e~\(\GK^\text{ab}\) de~\(\GK\). On note
\[
\rec_K\colon\AK\longrightarrow\GK^\text{ab}
\]
l'application de r\'eciprocit\'e de la th\'eorie du corps de classes globale et on pose
\[
\alpha=\widetilde{\blabla}\circ\rec_K\colon\AK\longrightarrow\C^\times.
\]
C'est un \gross\ de~\(K\) d'image finie, de conducteur~\(\cond_{\alpha}\) \'egal au conducteur d'Artin de~\(\blabla\) (voir~\cite[Chapitre~VI]{Ser68}, notamment page~110). 
D'apr\`es le Lemme~\ref{lem:cond_Artin} et la construction de~$\alpha$, on a
\begin{equation}\label{eq:N=D*norm}
N(\rho)
=D_K^{(\ell)}\norm\bigl(\mathfrak{N}(\blabla)\bigr)^{(\ell)}
=D_K^{(\ell)}\norm(\cond_{\alpha})^{(\ell)}.
\end{equation}

Soit~\(\delta\) le \gross\ de~\(K\) fourni par la Proposition~\ref{prop:TCC} appliqu\'ee au \gross\ \(\alpha\) ci-dessus et \`a l'entier~$k=k(\rho)$. On d\'efinit un caract\`ere, not\'e~$\delta_{\mathrm{H}}$, du groupe des id\'eaux fractionnaires de $K$ premiers \`a~$\cond_\delta$ de la fa\c con suivante~:
\[
\delta_{\mathrm{H}}(\mathfrak{p})=\delta(\pi_\mathfrak{p})\quad\text{pour tout premier }\mathfrak{p}\nmid \cond_\delta\text{ de }\mathcal{O}_K,
\]
o\`u l'on note encore $\pi_\mathfrak{p}\in\Afin$  l'id\`ele ayant composantes triviales en dehors de la place induite par~$\mathfrak{p}$ et qui co\"incide avec  l'uniformisante~$\pi_\mathfrak{p}$ choisie en cette place. 
Soit $\eta$ la fonction d\'efinie pour~$m\in\Z$ par
\[
\eta(m)=\frac{\delta_{\mathrm{H}}\left(m\mathcal{O}_K\right)}{m^{k(\rho)-1}},
\]
o\`u l'on convient que $\delta_{\mathrm{H}}\left(m\mathcal{O}_K\right)=0$ lorsque $m\mathcal{O}_K$ et~$\cond_\delta$ ne sont pas premiers entre eux. Alors, $\eta$ induit un caract\`ere de Dirichlet modulo~$M=\norm(\cond_\delta)$ (voir le Lemme~\ref{lemme:conducteurs} pour un r\'esultat plus pr\'ecis).

Soit enfin $\left(\frac{-D_K}{\cdot}\right)$ le symbole de Kronecker 
correspondant au corps~$K$ vu comme caract\`ere de Dirichlet 
modulo~$D_K$ et posons $\varepsilon=\left(\frac{-D_K}{\cdot}\right)\eta$, vu 
comme caract\`ere de Dirichlet modulo~$MD_K$. Pour~\(z\in\C\), de partie 
imaginaire~\(\Im(z)>0\), on pose alors
\[
g_{\delta}(z)=\sum_{\substack{(\mathfrak{a},\mathfrak{\cond_\delta})=1 \\ \mathfrak{a}\textrm{ entier}}}
\delta_{\mathrm{H}}(\mathfrak{a})q^{\norm(\mathfrak{a})}=\sum_{n\ge1}c_nq^n,\quad\text{avec }q=e^{2i\pi z}.
\]
Le r\'esultat suivant est d\^u \`a Hecke et Shimura (voir~\cite[page~717]{Hec59} ainsi que~\cite[Lemma~3]{Shi71b} et~\cite[page~138]{Shi72}).
\begin{theoreme}[Hecke, Shimura]\label{thm:Hecke}
La s\'erie $g_{\delta}$ est le d\'eveloppement de Fourier d'une newform de poids 
$k(\rho)$, niveau~$MD_K$ et caract\`ere~$\varepsilon$ \`a multiplication 
complexe par le corps~\(K\).
\end{theoreme}
Dans la d\'emonstration du Th\'eor\`eme~\ref{thm:main} ci-dessous, on note pour simplifier $k$ \`a la place de~$k(\rho)$. D'apr\`es la Proposition~\ref{prop:TCC}, $\cond_\alpha$ et~\(\cond_\delta\) co\"incident hors de~$\place$. En particulier, on a 
\begin{equation}\label{eq:cond=cond}
N(\rho)=D_K^{(\ell)}\norm(\cond_{\delta})^{(\ell)}=(MD_K)^{(\ell)}.
\end{equation}
On calcule \`a pr\'esent la valuation de~\(MD_K\) en~\(\ell\) afin de 
d\'eterminer l'entier $MD_K$ lui-m\^eme. On proc\`ede suivant les 
diff\'erents cas de la Proposition~\ref{prop:ramification} dont on reprend la 
terminologie et la notation.
\begin{enumerate}[label=(\alph*)]
\item On distingue deux cas~:
\begin{enumerate}[label=\ref{item:prop_ramification_niv1}-\arabic*.]
	\item Supposons \(\ell\) d\'ecompos\'e dans \(K\). D'apr\`es la Proposition~\ref{prop:ramification}, \(\blabla\) est mod\'er\'ement ramifi\'e en~\(\mathfrak{p}_\place\) et pour tout \(\tau\in I_{\mathfrak{p}_\place}\), on a \(\blabla(\tau)=\chi^{k-1}(\tau)\). Il suit que~$\ord_\place(\cond_\alpha)=1$ car il en est ainsi pour $\widetilde{\chi}\circ\rec_K$. Par ailleurs, le compl\'et\'e \(K_\place\) de \(K\) en \(\place\) s'identifie \`a~\(\Q_{\ell}\) et d'apr\`es \cite[Proposition~8]{Ser72} le caract\`ere fondamental~$\theta_{\ell-1}$ de niveau~\(1\) est le caract\`ere cyclotomique~\(\chi\). Soit $u\in\mathcal{O}_\place^\times$. On pose $x=\overline{u}\in\F_\ell^\times$. Avec les notations des paragraphes~\ref{subsec:teichmuller} et~\ref{ss:calculTCCL}, il d\'ecoule de l'\'egalit\'e~(\ref{eq:calculTCCL}) et de la compatibilit\'e entre les th\'eories du corps de classes locale et globale que l'on a 
	\begin{equation*}
\alpha_\place(u)=\widetilde{\blabla(\omega(x))}=\widetilde{\theta_{\ell-1}(\omega(x))}^{k-1}=\widetilde{x}^{1-k}=\widetilde{\overline{u}}^{1-k}.
\end{equation*}
D'apr\`es le point~\ref{condition:falpha=1} de la Proposition~\ref{prop:TCC}, on 
a donc~\(\ord_\place(\cond_\delta)=0\). D'o\`u il vient 
$\ord_{\ell}(MD_K)=0$ et $MD_K=N(\rho)$ d'apr\`es 
\eqref{eq:cond=cond}.
	\item Supposons \(\ell\) ramifi\'e dans \(K\), donc $\ell\mid D_K$. 
D'apr\`es la Proposition~\ref{prop:ramification}, \(\blabla\) est non ramifi\'e 
en~\(\mathfrak{p}_\place\) et \(\ell=2k-1\). On 
a~\(\ord_\place(\cond_\alpha)=0\) et $\ell(=2k-1)>k$. En particulier, on 
d\'eduit du point~\ref{condition:falpha<1} de la Proposition~\ref{prop:TCC} que 
l'on a~\(\ord_\place(\cond_\delta)=1\), puis \(\ord_{\ell}(MD_K)=2\) 
et~\(MD_K=\ell^2N(\rho)\).
	\end{enumerate}
\item D'apr\`es la Proposition~\ref{prop:ramification}, \(\ell\) est inerte ou ramifi\'e dans~\(K\), \(\blabla\) est ramifi\'e en~\(\mathfrak{p}_\place\) et pour tout \(\tau\in I_{\mathfrak{p}_\place}\), on a \(\blabla(\tau)=\psi_2^{k-1}(\tau)\), avec \(\psi_2\) caract\`ere fondamental de niveau~\(2\). Par suite, \(\blabla\) est mod\'er\'ement ramifi\'e et \(\ord_\place(\cond_{\alpha})=1\). Quitte \`a remplacer \(\blabla\) par \(\blabla'\), on peut de plus supposer que l'on a l'\'egalit\'e \(\blabla=\theta_{\ell^2-1}^{k-1}\) entre caract\`eres de~\(I_{\mathfrak{p}_\place}\). 
\begin{enumerate}[label=\ref{item:prop_ramification_niv2}-\arabic*.]
\item Supposons \(\ell\) inerte dans~\(K\). Dans ce cas, l'extension \(K_\place/\Q_{\ell}\) est quadratique non ramifi\'ee. Soit $u\in\mathcal{O}_\place^\times$. On pose $x=\overline{u}\in\F_{\ell^2}^\times$. D'apr\`es  l'\'egalit\'e~(\ref{eq:calculTCCL}), on a comme pr\'ec\'edemment, 
\begin{equation*}
\alpha_\place(u)=\widetilde{\blabla(\omega(x))}=\widetilde{\theta_{\ell^2-1}(\omega(x))}^{k-1}=\widetilde{x}^{1-k}=\widetilde{\overline{u}}^{1-k}.
\end{equation*}
D'apr\`es le point~\ref{condition:falpha=1} de la Proposition~\ref{prop:TCC}, on 
a donc~\(\ord_\place({\cond_\delta})=0\). D'o\`u il 
vient~\(MD_K=N(\rho)\).
\item Supposons enfin \(\ell\) ramifi\'e dans \(K\). Dans ce cas, l'extension \(K_\place/\Q_{\ell}\) est quadratique ramifi\'ee. Soit $u\in\mathcal{O}_\place^\times$. On pose $x=\overline{u}\in\F_\ell^\times$. Comme le produit des caract\`eres fondamentaux de niveau \(2\) est le caract\`ere fondamental de niveau \(1\), on d\'eduit comme ci-dessus des \'egalit\'es \(\ell=2k-3\) et~(\ref{eq:calculTCCL}) que l'on a 
\begin{align*}
\alpha_\place(u)^2&=\widetilde{\blabla(\omega(x))^2}=\widetilde{\theta_{\ell^2-1}^{2(k-1)}(\omega(x))}=\widetilde{\theta_{\ell^2-1}^{\ell+1}(\omega(x))} \\
&=\widetilde{\theta_{\ell-1}(\omega(x))}=\widetilde{x}^{-1}=\widetilde{\overline{u}}^{-1}.
\end{align*}
Or, on a 
$\widetilde{\overline{u}}^{2(1-k)}=\widetilde{x}^{-(\ell+1)}=\widetilde{x}^{-2}
=\widetilde{\overline{u}}^{-2}$, qui n'est pas \'egal \`a 
$\widetilde{\overline{u}}^{-1}$ d\`es lors que $\overline{u}\neq 1$~: en 
particulier~$\alpha_\place(u)\neq \widetilde{\overline{u}}^{1-k}$ d\`es lors que 
$u\in\mathcal{O}_\place^\times\setminus\mathcal{O}_v^{(1)}$. D'apr\`es le 
point~\ref{condition:falpha=1} de la Proposition~\ref{prop:TCC}, on a 
donc~\(\ord_\place({\cond_\delta})=1\). D'o\`u il 
vient~\(\ord_{\ell}(MD_K)=2\), puis~
\(MD_K=\ell^2 N(\rho)\). 
\end{enumerate}
\end{enumerate}
On a donc  montr\'e que l'on a
\begin{equation}\label{eq:DM=N_split}
MD_K=\left\{
\begin{array}{ll}
N(\rho) & \textrm{si \(\ell\) est non ramifi\'e dans \(K\)}; \\
\ell^2N(\rho) & \textrm{si \(\ell\) est ramifi\'e dans~\(K\)}. \\
\end{array}
\right.
\end{equation}

Prouvons \`a pr\'esent que $\rho$ provient bien de~$g_\delta$ (par r\'eduction modulo~$\place$)~: pour ce faire, nous allons v\'erifier que pour tout nombre premier $q\nmid N(\rho)\ell$, le polyn\^ome caract\'eristique de~$\rho(\Frob_q)$ est la r\'eduction de
\begin{equation*}
X^2-c_qX+\varepsilon(q)q^{k-1},
\end{equation*}
o\`u~\(\Frob_q\) est un repr\'esentant de Frobenius en~\(q\) dans~\(\GQ\).
\begin{description}
\item[Cas inerte] On suppose~\(q\) inerte dans~\(K\), de sorte qu'on peut 
choisir $\sigma=\Frob_q\in\GQ\setminus\GK$ au d\'ebut du 
paragraphe~\ref{subsec:descr_loc}. On pose $\mathfrak{q}=q\mathcal{O}_K$. 
On a d'une part les \'egalit\'es $c_q=0$ 
et~\(\varepsilon(q)q^{k-1}=-\delta_{\mathrm{H}}(\mathfrak{q})\). D'autre part, 
l'\'ecriture \eqref{eq:rho_sigma} entra\^ine~\({\tr \rho(\Frob_q)=0}\) ainsi que 
$\det\rho(\Frob_q)=-\blabla\left(\Frob_q^2\right)$. La th\'eorie du corps de 
classes locale et la Proposition~\ref{prop:TCC} donnent alors
\[
\blabla\left(\Frob_q^2\right)=\overline{\alpha(\pi_{\mathfrak{q}})}=\overline{\delta(\pi_{\mathfrak{q}})}=\overline{\delta_\mathrm{H}(\mathfrak{q})}.
\]
\item[Cas d\'ecompos\'e] On suppose~\(q\) d\'ecompos\'e dans~\(K\) et on note~\(\mathfrak{q},\mathfrak{q}'\) les id\'eaux premiers de~\(\entiers[K]\) au-dessus de~\(q\). Alors $\Frob_q\in\GK$ est une substitution de Frobenius en~$\mathfrak{q}$ et en~$\mathfrak{q'}$. On a, d'une part, $c_q=\delta_{\mathrm{H}}(\mathfrak{q})+\delta_{\mathrm{H}}(\mathfrak{q}')$ et \(\varepsilon(q)q^{k-1}=\delta_{\mathrm{H}}(\mathfrak{q}\mathfrak{q'})\). D'autre part, on a comme pr\'ecedemment, \[
\tr\rho(\Frob_q)=\blabla(\Frob_q)+\widehat{\blabla}(\Frob_q)
=\overline{\alpha(\pi_{\mathfrak{q}})}+\overline{\alpha(\pi_{\mathfrak{q}'})}
=\overline{\delta(\pi_{\mathfrak{q}})}+\overline{\delta(\pi_{\mathfrak{q}'})}=\overline{c_q}
\]
et
\[
\det\rho(\Frob_q)=\blabla(\Frob_q)\widehat{\blabla}(\Frob_q)
=\overline{\alpha(\pi_{\mathfrak{q}})}\, \overline{\alpha(\pi_{\mathfrak{q}'})}
=\overline{\delta(\pi_{\mathfrak{q}})}\, \overline{\delta(\pi_{\mathfrak{q}'})}
=\overline{\delta_{\mathrm{H}}(\mathfrak{q}\mathfrak{q'})}.
\]
\end{description}
On en d\'eduit le r\'esultat voulu avec~\cite[Lemme~3.2]{DeSe74}.

Pour terminer la d\'emonstration du Th\'eor\`eme~\ref{thm:main}, il reste \`a v\'erifier l'affirmation sur le caract\`ere. Soit $\varepsilon(\rho)$ le caract\`ere associ\'e \`a~$\rho$ par Serre comme dans~\cite[\no1.3]{Ser87}. On fait l'hypoth\`ese que~$\varepsilon(\rho)$ est trivial, i.e.~$\det\rho=\chi^{k-1}$ avec $\chi$ le caract\`ere cyclotomique mod~$\ell$. On doit alors montrer qu'il existe une forme \`a multiplication complexe v\'erifiant les conditions du th\'eor\`eme et dont le caract\`ere est trivial. Pour cela, on va consid\'erer une tordue de~$g_\delta$ par un caract\`ere de Dirichlet particulier.
On a la d\'ecomposition suivante de l'entier~$M$~:
\[
M=\norm(\cond_\delta)  =\prod_{\fp}\norm(\fp)^{\ord_\fp(\cond_\delta)}  = \prod_p p^{\sum_{\fp\mid p}f_{\fp/p}\ord_\fp(\cond_\delta)},
\]
o\`u $f_{\fp/p}$ d\'esigne le degr\'e r\'esiduel de l'id\'eal premier~$\fp$ de~$\mathcal{O}_K$ au-dessus du premier~$p$.
\begin{lemme}\label{lemme:v_p(M)}
Lorsque $\det\rho=\chi^{k-1}$ la valuation~$\ord_p(M)$ de~$M$ en tout nombre 
premier~$p$ ne divisant pas~$D_K$ est paire, \'egale \`a 
$2\ord_\fp(\cond_\delta)$ o\`u $\fp$ est un premier de $\mathcal{O}_K$ au-dessus 
de $p$.
\end{lemme}
\begin{proof}
Soit $p$ ne divisant pas~$D_K$ tel que~$\ord_p(M)>0$. Lorsque $p$ est 
inerte dans~$K$, avec $p\mathcal{O}_K=\fp$, le r\'esultat est imm\'ediat car 
alors $\ord_p(M)=2\ord_\fp(\cond_\delta)$. On suppose donc \`a pr\'esent $p$ 
d\'ecompos\'e dans~$K$, disons $p\mathcal{O}_K=\fp\bar{\fp}$. On note que $p$ 
est n\'ecessairement diff\'erent de~$\ell$ car lorsque~$\ell$ est d\'ecompos\'e 
dans~$K$, alors, gr\^ace \`a~\eqref{eq:DM=N_split}, on a 
$MD_K=N(\rho)$ qui est premier \`a~$\ell$. 
D'apr\`es~\eqref{eq:N=D*norm} on a 
$\ord_\fp(\cond_\delta)=\ord_\fp(\mathfrak{N}(\blabla))$, et de m\^eme 
pour~$\bar{\fp}$. Soit $I_\fp$ un groupe d'inertie en~$\fp$ dans~$\GK$. 
D'apr\`es l'hypoth\`ese $\det\rho=\chi^{k-1}$ et avec les notations du 
paragraphe~\ref{subsec:descr_loc}, on a 
$\blabla_{|I_\fp}=\widehat{\blabla}^{-1}_{|I_\fp}$ parce que 
$\chi_{|I_{\fp}}=1$. D'o\`u il vient 
\[
\ord_\fp(\mathfrak{N}(\blabla))
=\ord_\fp\left(\mathfrak{N}\left(\widehat{\blabla}^{-1}\right)\right)
=\ord_\fp\left(\mathfrak{N}\left(\widehat{\blabla}\right)\right)
=\ord_{\bar{\fp}}\left(\mathfrak{N}\left(\blabla\right)\right)
\]
puis $\ord_\fp(\cond_\delta)=\ord_{\bar{\fp}}(\cond_\delta)$ et $\ord_p(M)=\ord_\fp(\cond_\delta)+\ord_{\bar{\fp}}(\cond_\delta)=2\ord_\fp(\cond_\delta)$.
\end{proof}
On pose alors
\[
M'=\prod_{p\nmid D_K}p^{\ord_p(M)/2}\times \prod_{p\mid D_K}p^{x_p},
\]
o\`u pour tout nombre premier~$p$ divisant~$D_K$, on d\'efinit
\[
x_p=\left\{
\begin{array}{ll}
\ord_p(M)/2 & \text{si $\ord_p(M)$ est pair;} \\
(\ord_p(M)+1)/2 & \text{si $\ord_p(M)$ est impair.} \\
\end{array}
\right.
\]
D'apr\`es le lemme pr\'ec\'edent, c'est un entier naturel divisant~$M$ qui v\'erifie 
$\ord_\fp(M')\geq\ord_\fp(\cond_\delta)$ pour tout $\fp\mid \cond_\delta$: cela 
suit du Lemme \ref{lemme:v_p(M)} pour les $\fp$ premiers \`a $D_K$, et pour 
ceux qui divisent $D_K$ du calcul $\ord_{\fp}(M')=2\ord_p(M')\geq 
\ord_p(M)=\ord_{\fp}(\cond_\delta)$, o\`u on a not\'e $p=\norm(\fp)$.
\begin{lemme}\label{lemme:conducteurs}
Le caract\`ere $\varepsilon$ se factorise modulo~$M'$.
\end{lemme}
\begin{proof}
On reprend les notations du d\'ebut de ce paragraphe et on commence par 
comparer~$m^{k-1}\eta(m)$ avec $\delta_{\mathrm{f}}(m)=m^{k-1}$ pour un entier 
$m\equiv1\pmod{M'}$~:  
$M'$ et~$M$ ont les m\^emes diviseurs premiers, on a en particulier $\pgcd(m,M)=1$. On remarque tout 
d'abord que la formule
\[
\delta_{\mathrm{H}}(\fp)^{\ord_\fp(m\mathcal{O}_K)}=\delta(\pi_\fp)^{\ord_\fp(m\mathcal{O}_K)}
\]
est valable pour tout id\'eal premier~$\fp$ de~$\mathcal{O}_K$. En effet, lorsque $\ord_\fp(m\mathcal{O}_K)=0$, elle est triviale et lorsque $\ord_\fp(m\mathcal{O}_K)>0$, cela r\'esulte du fait que cela entra\^ine $\fp\nmid M=\norm(\cond_\delta)$ et par suite la formule n'est autre chose que la d\'efinition de $\delta_{\mathrm{H}}$.
Ainsi, on a 
\[
m^{k-1}\eta(m)=\delta_{\mathrm{H}}(m\mathcal{O}_K)=\prod_\fp\delta_{\mathrm{H}}(\fp)^{\ord_\fp(m\mathcal{O}_K)}=\prod_\fp\delta(\pi_\fp)^{\ord_\fp(m\mathcal{O}_K)}
\]
d'o\`u l'on d\'eduit $m^{k-1} =\delta_{\mathrm{f}}(m) = \prod_{\fp\mid m\mathcal{O}_K}\delta_\fp(m)\times \prod_{\fp\nmid m\mathcal{O}_K}\delta_\fp(m)$, puis
\begin{align*}
 m^{k-1} & =\prod_{\fp\mid m\mathcal{O}_K}\delta_\fp(\pi_\fp)^{\ord_\fp(m\mathcal{O}_K)}\times \prod_{\fp\nmid m\mathcal{O}_K}\delta_\fp(m) \\
 & =\prod_{\fp}\delta_\fp(\pi_\fp)^{\ord_\fp(m\mathcal{O}_K)}\times \prod_{\fp\nmid m\mathcal{O}_K}\delta_\fp(m)
  =m^{k-1}\eta(m)\times \prod_{\fp\nmid m\mathcal{O}_K}\delta_\fp(m).
\end{align*}
Notons $P(m)=\prod_{\fp\nmid m\mathcal{O}_K}\delta_\fp(m)=\eta(m)^{-1}$ le produit apparaissant \`a la derni\`ere ligne de la formule ci-dessus. Tout d'abord, on remarque que l'on a 
$\displaystyle{P(m)=\prod_{\fp\mid \cond_\delta}\delta_\fp(m)}$  car d'une part $\pgcd(\cond_\delta,m\mathcal{O}_K)=1$ et d'autre part, si $\fp\nmid m\mathcal{O}_K$ et~$\fp\nmid\cond_\delta$, alors $m\in\mathcal{O}_\fp^\times$, puis $\delta_\fp(m)=1$. 
Comme pour tout $\fp\mid\cond_\delta$ on a $\ord_\fp(m-1)\geq\ord_\fp(M')\geq\ord_\fp(\cond_\delta)$, 
par d\'efinition du conducteur d'un \gross, on en d\'eduit $P(m)=1$, puis $\eta(m)=1$. Maintenant on remarque que l'on a
\[
\varepsilon(m)=\eta(m) \left(\frac{-D_K}{m}\right)= 
\left(\frac{-D_K}{m}\right)=\pm1
\]
car $\eta$ se factorise modulo~$M'$. Or, $\varepsilon(m)$ se r\'eduit sur~$1$ modulo~$\ell$ par hypoth\`ese. On a donc $\varepsilon(m)=1$ et le r\'esultat voulu.
\end{proof}

Le caract\`ere $\varepsilon$, se r\'eduisant sur le caract\`ere trivial, est n\'ecessairement d'ordre une puissance de~$\ell$, disons~$\ell^h$. On pose alors~$\mu=\varepsilon^{(\ell^h-1)/2}$. C'est un caract\`ere de m\^eme ordre et de m\^eme conducteur que~$\varepsilon$ v\'erifiant~$\mu^2=\varepsilon^{-1}$. 
Pour~\(z\in\C\), de partie imaginaire~\(\Im(z)>0\), on pose alors
\[
g^\dag(z)=(g_\delta\otimes\mu)(z)=\sum_{n\geq1}\mu(n)c_nq^n,\quad\text{avec }q=e^{2i\pi z}.
\]
D'apr\`es~\cite[Proposition 3.64]{Shi71a}, $g^\dag$ est une forme propre de 
poids~$k$, de caract\`ere $\varepsilon\mu^2=1$ et de niveau~
$\ppcm(MD_K,r^2)$ o\`u $r$ est le conducteur de~$\varepsilon$ (ou 
de~$\mu$). Or, on v\'erifie que $M'^2$ divise le produit~$MD_K$. Ainsi, 
$\ppcm(MD_K,r^2)=MD_K$ d'apr\`es le Lemme~\ref{lemme:conducteurs}. 
Soit $g$ la newform associ\'ee \`a~$g^\dag$. Alors, $g$ est de niveau 
divisant~$MD_K$ et \`a multiplication complexe par~$K$. Comme~$\mu$ se 
r\'eduit sur le caract\`ere trivial, la r\'eduction de la repr\'esentation 
$\place$-adique associ\'ee \`a~$g$ est isomorphe \`a celle de $g_\delta$, puis 
\`a~$\rho$. Ceci ach\`eve la d\'emonstration du Th\'eor\`eme~\ref{thm:main}.

%
%

\section{Exemples num\'eriques}\label{sec:exemples}

Pour l'exemple ci-dessous nous avons utilis\'e le logiciel \verb'pari/gp' (\cite{PARI2}).

\subsection{}\label{subsec:Delta} Soit $\Delta=\sum_{n\geq1}\tau(n)q^n=q-24q^2+252q^3-1472q^4+4830q^5+\cdots$ l'unique forme parabolique normalis\'ee de poids~$12$ et de niveau~$1$. On note
\[
\rho_{\Delta,23}\colon\GQ\longrightarrow\GL_2(\F_{23})
\]
l'unique repr\'esentation galoisienne semi-simple non ramifi\'ee hors de~$23$ telle que
\[
\det \rho_{\Delta,23}(\Frob_p)=\tau(p)\pmod{23}\quad\text{et}\quad\tr\ \rho_{\Delta,23}(\Frob_p)=p^{11}\pmod{23}
\]
pour tout nombre premier~$p\not=23$. Elle est de poids~$k(\rho_{\Delta,23})=12$ et de conducteur~$N(\rho_{\Delta,23})=1$. 

Il r\'esulte des congruences de Wilton (\cite[page~2]{Wil30}) que la repr\'esentation~$\rho_{\Delta,23}$ est di\'edrale. Soit, plus explicitement, $H$ le corps de classes de Hilbert de $K=\Q(\sqrt{-23})$. C'est une extension de degr\'e~$6$ de~$\Q$ de groupe de Galois~$D_3$. Alors, $\rho_{\Delta,23}$ est isomorphe \`a la repr\'esentation
\[
\GQ\longrightarrow\Gal(H/\Q)\simeq D_3\stackrel{\sigma}{\longrightarrow}\GL_2(\Z)\longrightarrow\GL_2(\F_{23}),
\]
o\`u la derni\`ere fl\`eche est l'application de r\'eduction modulo~$23$ et $\sigma\colon D_3\rightarrow\GL_2(\Z)$ est l'unique repr\'esentation irr\'eductible de degr\'e~$2$ du groupe~$D_3$ (voir~\cite[\S 3.4]{Ser69}).

On explicite \`a pr\'esent la forme \`a multiplication complexe dont l'existence 
est garantie par le Th\'eor\`eme~\ref{thm:main}. Soit $\delta$ un \gross\ de~$K$ 
de type \`a l'infini~$(11,0)$ et conducteur~$\sqrt{-23}\mathcal{O}_K$ de sorte 
que, avec les notations 
du paragrahe~\ref{subsec:dem_thm_main}, 
$\eta=\left(\frac{-23}{\cdot}\right)$. La newform~$g_\delta$ associ\'ee 
\`a~$\delta$ est alors de poids~$12$, niveau~$23^2$, caract\`ere trivial et de corps des coefficients l'extension totalement r\'eelle~$L$  de degr\'e~$3$ de~$\Q$ engendr\'ee par une 
racine~$\alpha$ du polyn\^ome~$X^3-6X-3$ (pour les d\'etails, voir le fichier 
\verb'DeltaMod23.gp', disponible sur la page de pr\'epublication arXiv de ce travail). On a
\begin{equation*}
g_\delta=\sum_{n\geq1}c_nq^n= q+(-21\alpha^2 - 4\alpha + 84) q^2+\left(53\alpha^2 + 251\alpha - 212\right)q^3+\cdots
\end{equation*}

On pose
\[\Delta^\dag=\sum_{\substack{n\geq1 \\ 23\nmid n}}\tau(n)q^n=\sum_{n\geq1}\tau'(n)q^n
\] de sorte que 
\(
\Delta^\dag(z)=\Delta(z)-U_{23}(\Delta)(23z)\)
o\`u $U_{23}$ est l'op\'erateur de Hecke en~$23$ agissant sur les formes de poids~$12$ et de niveau divisible par~$23$. Ainsi, $\Delta^\dag$ et~$g_\delta$ sont des formes normalis\'ees de poids~$12$, niveau~$23^2$ et caract\`ere trivial qui sont propres pour l'alg\`ebre de Hecke engendr\'ee par les op\'erateurs~\(\left\{T_p,\ p\text{ premier },\ p\not=23\right\}\).

Soit~$\lambda_{23}=(\alpha-5)\mathcal{O}_L$ l'unique id\'eal premier de l'anneau~$\mathcal{O}_L$ des entiers de~$L$ ramifi\'e au-dessus de~$23$. On v\'erifie alors que l'on a
\[
\tau'(n)\equiv c_n\pmod{\lambda_{23}},\quad\text{pour tout entier }n\leq m,
\]
o\`u $m=\left[\mathrm{SL}_2(\Z):\Gamma_0\left(23^2\right)\right]=552$. Pour les entiers~$n$ divisibles par~$23$, cela r\'esulte imm\'ediatement des \'egalit\'es~$\tau'(n)=0$ et~$c_n=0$. Lorsque $23\nmid n$, on se ram\`ene par multiplicativit\'e des coefficients de Fourier au cas o\`u $n$ est premier et on utilise la commande \verb'CoefficientFormeCM' du fichier \verb'DeltaMod23.gp' (l'ensemble des calculs prend quelques secondes sur un ordinateur de bureau). D'apr\`es~\cite[Theorem~1]{Stu87}, il vient donc \(\tau'(n)\equiv c_n\pmod{\lambda_{23}}\), pour tout entier $n\geq1$, puis 
\[
\tau(p)\equiv c_p\pmod{\lambda_{23}},\quad\text{pour tout nombre premier }p\not=23.
\]
On conclut avec~\cite[Lemme~3.2]{DeSe74} que~$\rho_{\Delta,23}$ est isomorphe \`a la r\'eduction modulo~$\lambda_{23}$ de la repr\'esentation galoisienne $23$-adique associ\'ee \`a~$g_\delta$. En particulier, comme l'affirme le Th\'eor\`eme~\ref{thm:main}, $\rho_{\Delta,23}$ provient bien d'une forme \`a multiplication complexe de poids~$k(\rho_{\Delta,23})=12$ et de niveau $N'=23^2N(\rho_{\Delta,23})=23^2$.

\begin{remarque}
On constate ici que le niveau~$23^2$ est optimal. En effet, il n'existe pas de forme \`a multiplication complexe de poids~$12$ et de niveau~$1$ ou~$23$ congrue \`a~$\Delta$ modulo~$23$.
\end{remarque}

\subsection{} Le Corollaire~\ref{cor} s'applique \'egalement \`a la repr\'esentation modulo~$7$ attach\'ee \`a la courbe elliptique d'\'equation
\(
Y^2+ Y = X^3 - X^2 - 18507X - 989382
\)
not\'ee \href{http://www.lmfdb.org/EllipticCurve/Q/65533/a/1}{65533.a1} dans \verb'LMFDB' (\cite{lmfdb}).  La d\'etermination de la forme \`a multiplication complexe (de niveau~$71^2$) dont l'existence est garantie par le Corollaire~\ref{cor} s'effectue par une m\'ethode analogue \`a celle utilis\'ee pr\'ec\'edemment. Les d\'etails sont accessibles dans le fichier \verb'EMod7.gp', disponible sur la page de pr\'epublication arXiv de ce travail.

\end{document}